\documentclass[graybox]{svmult}

\usepackage{xcolor}

\usepackage{type1cm}        
\usepackage{makeidx}         
\usepackage{graphicx}        
\usepackage{multicol}        
\usepackage[bottom]{footmisc}

\usepackage{enumitem}

\usepackage{newtxtext}       %
\usepackage{newtxmath} 

\makeindex             

\begin{document}

\title*{Three proofs of the Benedetto--Fickus theorem}
\author{Dustin G.\ Mixon, Tom Needham, Clayton Shonkwiler, Soledad Villar}
\institute{Dustin G.\ Mixon \at The Ohio State University, Columbus, OH, \email{mixon.23@osu.edu}
\and Tom Needham \at Florida State University, Tallahassee, FL, \email{tneedham@fsu.edu}
\and Clayton Shonkwiler \at Colorado State University, Fort Collins, CO, \email{clayton.shonkwiler@colostate.edu}
\and Soledad Villar \at Johns Hopkins University, Baltimore, MD, \email{soledad.villar@jhu.edu}}
%
%
\maketitle


\abstract{In 2003, Benedetto and Fickus introduced a vivid intuition for an objective function called the frame potential, whose global minimizers are fundamental objects known today as unit norm tight frames.
Their main result was that the frame potential exhibits no spurious local minimizers, suggesting local optimization as an approach to construct these objects.
Local optimization has since become the workhorse of cutting-edge signal processing and machine learning, and accordingly, the community has identified a variety of techniques to study optimization landscapes.
This chapter applies some of these techniques to obtain three modern proofs of the Benedetto--Fickus theorem.}

\section{Introduction}
\label{sec.intro}

Frame theory is the study of overcomplete systems of vectors in a Hilbert space.
Much of the last 20 years of research in finite frame theory has concerned various modifications of the following fundamental object:
(the sequence of column vectors of) $Z=[z_1\cdots z_n]\in\mathbb{C}^{d\times n}$ is a \textbf{unit norm tight frame} if 
\begin{itemize}
\item
$\|z_j\|^2=1$ for every $j\in[n]:=\{1,\ldots,n\}$, and 
\item
$ZZ^*$ is a multiple of the $d\times d$ identity matrix.
\end{itemize}
The multiple is necessarily the \textbf{redundancy} $\frac{n}{d}$ of $Z$ since $\operatorname{tr}(ZZ^*)=\operatorname{tr}(Z^*Z)=n$.
Given a $d\times n$ unit norm tight frame, it necessarily holds that
\[
n
\geq \operatorname{rank}(Z)
=\operatorname{rank}(ZZ^*)
=\operatorname{rank}(\tfrac{n}{d}I_d)
=d.
\]
The unit norm tight frames with $n=d$ are precisely the orthonormal bases for $\mathbb{C}^d$.
When $n>d$, a unit norm tight frame $Z\in\mathbb{C}^{d\times n}$ provides both an encoding $y:=Z^*x$ of data $x\in\mathbb{C}^d$ that is robust to various types of errors~\cite{GoyalKK:01,CasazzaK:03,HolmesP:04} (thanks in part to the redundancy $\frac{n}{d}>1$) and a painless reconstruction formula
\[
x
=\frac{d}{n}ZZ^*x
=\frac{d}{n}Zy
\]
(thanks to the resemblance $ZZ^*=\frac{n}{d}I_d$ to orthonormal bases).

Historically, frames were introduced in the setting of infinite-dimensional Hilbert space by Duffin and Schaeffer~\cite{DuffinS:52} in 1952 in the development of non-harmonic Fourier series.
In 1986, Daucechies, Grossmann, and Meyer~\cite{DaubechiesGY:86} further developed the theory of frames in the context of signal processing.
In the time since, finite-dimensional treatments of the subject have proved interesting in their own right.

Today, certain unit norm tight frames are known to emerge as optimal frames under various objectives, for example, when hunting for $n$ points on the unit sphere in $\mathbb{C}^d$ for which the absolute values of the inner products between pairs of points are uniformly as small as possible~\cite{Welch:74,StrohmerH:03,FickusM:15}.
In this chapter, we focus on an objective for which the $d\times n$ unit norm tight frames are \textit{precisely} the global minimizers.
Let $\operatorname{S}(d,n)$ denote the set of matrices $Z\in\mathbb{C}^{d\times n}$ for which every column of $Z$ has unit norm.
(We use $\operatorname{S}$ since each column resides in the \textit{sphere}.)
The \textbf{frame potential} $\operatorname{FP}\colon\operatorname{S}(d,n)\to\mathbb{R}$ is defined by
\[
\operatorname{FP}(Z):=\|Z^*Z\|_F^2,
\]
where $\|\cdot\|_F$ denotes the Frobenius matrix norm.

\begin{proposition}[Welch~\cite{Welch:74}]
For every $Z\in\operatorname{S}(d,n)$, it holds that $\operatorname{FP}(Z)\geq\frac{n^2}{d}$, with equality precisely when $Z$ is a unit norm tight frame.
\end{proposition}

\begin{proof}
Rearrange the inequality
\[
0
\leq\|ZZ^*-\tfrac{n}{d}I_d\|_F^2
=\|ZZ^*\|_F^2-2\langle ZZ^*,\tfrac{n}{d}I_d\rangle+\|\tfrac{n}{d}I_d\|_F^2
=\operatorname{FP}(Z)-\tfrac{n^2}{d},
\]
and observe that equality holds precisely when $ZZ^*=\frac{n}{d}I_d$.
\smartqed
\end{proof}

We note that when $n<d$, Welch's bound is dominated by the simpler bound
\[
\operatorname{FP}(Z)\geq n,
\]
and equality is achieved precisely when the columns of $Z$ are orthonormal.

In~\cite{BenedettoF:03}, Benedetto and Fickus introduced the idea of locally minimizing the frame potential, much like how point charges confined to a sphere will dynamically repel each other in order to locally minimize a Coulomb potential.
For context, there are several longstanding open problems that concern minimizing a given potential of a point configuration on the sphere, to include the Thomson problem~\cite{Thomson:04}, the Tammes problem~\cite{Tammes:30}, and Smale's 7th problem~\cite{Smale:98}.
The difficulty with these problems stems from what appear to be exponentially many spurious local minimizers, i.e., local minimizers that are not global minimizers.
Perhaps surprisingly, the frame potential does not share this defect:

\begin{proposition}[Benedetto--Fickus theorem~\cite{BenedettoF:03}]
For each $d,n\in\mathbb{N}$, the frame potential $\operatorname{FP}\colon\operatorname{S}(d,n)\to\mathbb{R}$ has no spurious local minimizers.
\end{proposition}

We offer two comments about the results in~\cite{BenedettoF:03}.
First, Benedetto and Fickus consider both the real and complex cases.
We isolate the complex case here because it appears to be more amenable to analysis by a variety of interesting techniques.
Throughout, we discuss to what extent certain arguments can be transferred to the real case.
Second, even if we were to include the real case, Benedetto and Fickus prove a bit more than what is contained in the Benedetto--Fickus theorem.
Specifically, they use Lagrange multipliers to characterize the critical points of the frame potential, and then they find a descent direction for every critical point that is neither orthonormal (in the case where $n\leq d$) nor a unit norm tight frame (in the case where $n\geq d$).
Note that by compactness, a global minimum necessarily exists, which in turn is a critical point with no descent direction, and by the above, this point must be a unit norm tight frame whenever $n\geq d$.
By this argument, Benedetto and Fickus provide a clever implicit proof of the existence of unit norm tight frames and the achievability of Welch's bound.

Regardless, we isolate the Benedetto--Fickus theorem as the main result of~\cite{BenedettoF:03}.
One reason is that this feature of the frame potential clearly distinguishes it from other potentials such as the Coulomb potential.
Second, there are much shorter explicit proofs of the existence of unit norm tight frames (even in the real case): for each $d,n\in\mathbb{N}$ with $n\geq d$, it suffices to select $d$ appropriate rows of the $n\times n$ real discrete Fourier transform matrix and scale the resulting columns to have unit norm.
Indeed, the frame potential is unusual in how easy it is to characterize its global minimizers.
Finally, our statement of the Benedetto--Fickus theorem highlights the relevance of~\cite{BenedettoF:03} to modern trends in signal processing and machine learning.
In particular, much of the last decade of research has centered around analyzing non-convex landscapes for optimization in various settings, such as orthogonal tensor decomposition~\cite{GeHJY:15}, the Burer--Monteiro method~\cite{BoumalVB:16}, dictionary learning~\cite{SunQW:16}, matrix completion~\cite{GeLM:16}, phase retrieval~\cite{SunQW:18}, and neural network training~\cite{LaurentB:18}.
In all of these settings, a primary goal has been to establish the absence of spurious local minimizers so as to explain the empirical performance of local optimization.

The recent literature has introduced many techniques to our community, and in this chapter, we use some of them to prove three different strengthenings of the Benedetto--Fickus theorem.
The proof in the next section is closest to the original, though we use a stronger Lagrange multiplier result to make the proof more routine, and we use Wirtinger derivatives to simplify the necessary calculations.
In Section~3, we apply a general observation that a convex function of an open map has no spurious local minimizers, and we obtain such a factorization of the frame potential using the theory of \textit{eigensteps} developed in~\cite{Cahill:13}.
Finally, Section~4 leverages ideas from geometric invariant theory to show that locally minimizing the frame potential by gradient flow sends almost every initialization to a unit norm tight frame.

\section{Lagrange multipliers and Wirtinger derivatives}
\label{sec.duality}

Lagrange multipliers provide a standard approach for optimization subject to equality constraints.
What follows is a fundamental result in this setting:

\begin{proposition}[Lagrange multiplier theorem, Proposition~3.1.1 in~\cite{Bertsekas:99}\label{prop.lagrange mult}]
Given twice continuously differentiable functions $f,h_1,\ldots,h_m\colon\mathbb{R}^n\to\mathbb{R}$, consider
\[
S:=\{x\in\mathbb{R}^n:h_1(x)=\cdots=h_m(x)=0\},
\]
and suppose $x_0\in S$ is a local minimum of the restriction of $f$ to $S$ with linearly independent constraint gradients $\{\nabla h_i(x_0)\}_{i\in[m]}$.
Then there exists a unique choice of Lagrange multipliers $\lambda_1,\ldots,\lambda_m\in\mathbb{R}$ such that
\begin{equation}
\label{eq.lagrange mult}
\nabla f(x_0)+\sum_{i\in[m]}\lambda_i\nabla h_i(x_0)
=0,
\quad
y^\top\bigg(\nabla^2 f(x_0)+\sum_{i\in[m]}\lambda_i\nabla^2 h_i(x_0)\bigg)y
\geq0
\end{equation}
for all $y\in(\operatorname{span}\{\nabla h_i(x_0)\}_{i\in[m]})^\perp$.
\end{proposition}

This motivates a couple of definitions.
In the setting of Proposition~\ref{prop.lagrange mult}, we say $x_0\in S$ is a \textbf{(first-order) critical point} if the constraint gradients $\{\nabla h_i(x_0)\}_{i\in[m]}$ are linearly dependent or there exist Lagrange multipliers $\lambda_1,\ldots,\lambda_m\in\mathbb{R}$ such that
\[
\nabla f(x_0)+\sum_{i\in[m]}\lambda_i\nabla h_i(x_0)
=0,
\]
and we say $x_0\in S$ is a \textbf{second-order critical point} if $\{\nabla h_i(x_0)\}_{i\in[m]}$ are linearly dependent or there exist $\lambda_1,\ldots,\lambda_m\in\mathbb{R}$ such that \eqref{eq.lagrange mult} holds for all $y\in(\operatorname{span}\{\nabla h_i(x_0)\}_{i\in[m]})^\perp$.
By Proposition~\ref{prop.lagrange mult}, every local minimizer is a second-order critical point, and by relaxation, every second-order critical point is a first-order critical point.
(Indeed, we include the dependence of constraint gradients in the definition of critical points so that this implication holds, but as we will see, such degeneracy never arises in our setting.)

In this section, we simplify the original proof of the Benedetto--Fickus theorem to obtain a proof of the following strengthening:

\begin{theorem}
\label{thm.socp}
For each $d,n\in\mathbb{N}$, the second-order critical points of the frame potential $\operatorname{FP}\colon\operatorname{S}(d,n)\to\mathbb{R}$ are precisely the global minimizers of the frame potential, namely, the $d\times n$ matrices with orthonormal columns when $n\leq d$, and the $d\times n$ unit norm tight frames when $n\geq d$.
\end{theorem}

While we will prove Theorem~\ref{thm.socp} in the complex setting, the same proof works (and is even simpler) in the real setting.
In our setting, we are optimizing the restriction of a function $f\colon\mathbb{C}^n\to\mathbb{R}$, and so to find critical points, we are inclined to take $g\colon\mathbb{R}^n\times\mathbb{R}^n\to\mathbb{C}^n$ defined by $g(x,y):=x+\mathrm{i}y$ and then compute the gradient of the composition $f\circ g$.
Working with this composition can be cumbersome, but as we will see, Wirtinger derivatives nicely encode the desired gradient information in terms of functions over the native space $\mathbb{C}^n$.
This formalism drew widespread attention from the signal processing community when Cand\`{e}s, Li, and Soltanolkotabi used it to analyze the phase retrieval problem in~\cite{CandesLS:15}.
The \textbf{Wirtinger derivative} $\frac{\partial f}{\partial \overline{z}_i}\colon\mathbb{C}^n\to\mathbb{C}$ is defined by
\[
\frac{\partial f}{\partial \overline{z}_i}
:=\frac{1}{2}\bigg(\frac{\partial(f\circ g)}{\partial x_i}+\mathrm{i}\frac{\partial(f\circ g)}{\partial y_i}\bigg)\circ g^{-1},
\]
which in turn determines the \textbf{Wirtinger gradient} $\frac{\partial f}{\partial \overline{z}}\colon\mathbb{C}^n\to\mathbb{C}^n$ defined by
\[
\frac{\partial f}{\partial \overline{z}}:=\bigg(\frac{\partial f}{\partial \overline{z}_1},\ldots,\frac{\partial f}{\partial \overline{z}_n}\bigg).
\]
For example, given $f\colon\mathbb{C}^n\to\mathbb{R}$ defined by $f(z):=\|z\|^2$, we have $(f\circ g)(x,y)=\sum_{j\in[n]}(x_j^2+y_j^2)$, and so $\frac{\partial(f\circ g)}{\partial x_i}(x,y)=2x_i$, $\frac{\partial(f\circ g)}{\partial y_i}(x,y)=2y_i$, and $\frac{\partial f}{\partial \overline{z}}(z)=z$.
Next, the \textbf{gradient} $\nabla f\colon\mathbb{C}^n\to\mathbb{C}^n$ is defined in terms of the best real linear approximation of $f$ at the input $z$:
\[
f(z+\Delta z)
=f(z)+\operatorname{Re}\langle \nabla f(z),\Delta z\rangle+\text{h.o.t.}
\]
Since $g$ is a real linear isometry, it follows that $\nabla (f\circ g)=g^{-1}\circ\nabla f\circ g$, meaning $\nabla f$ is the notion of gradient that is relevant to optimization.
One may verify that
\[
\nabla f=2\frac{\partial f}{\partial \overline{z}}.
\]
For example, for $f(z):=\|z\|^2$, we have $\nabla f(z)=2z$, which nicely resembles the gradient expression when $z$ is real.
The most attractive aspect of Wirtinger derivatives is that if $f(z)$ is a polynomial in $z$ and $\overline{z}$, then $\frac{\partial f}{\partial \overline{z}_i}(z)$ may be computed by formally differentiating in $\overline{z}_i$ as if everything else (including $z_i$) is constant.
For example, for $f(z):=\|z\|^2=\sum_{j\in[n]}z_j\overline{z}_j$, we have $\frac{\partial f}{\partial \overline{z}_i}(z)=z_i$, matching our previous calculation.

\begin{lemma}
\label{lem.euclidean gradient}
Consider the extension $\operatorname{EP}\colon\mathbb{C}^{d\times n}\to\mathbb{R}$ of the frame potential defined by $\operatorname{EP}(Z):=\|Z^*Z\|_F^2$, and for each $j\in[n]$, consider the constraint $h_j\colon\mathbb{C}^{d\times n}\to\mathbb{R}$ defined by $h_j(Z):=\|Ze_j\|^2-1$, where $e_j$ denotes the $j$th standard basis element in $\mathbb{C}^n$.
Then for $Z=[z_1\cdots z_n]$, it holds that
\[
\nabla\operatorname{EP}(Z)=4ZZ^*Z,
\qquad
\nabla h_j(Z)=2z_je_j^*.
\]
\end{lemma}

\begin{proof}
For $\operatorname{EP}$, we expand
\begin{align*}
\operatorname{EP}(Z)
=\|Z^*Z\|_F^2
=\sum_{i\in[n]}\sum_{j\in[n]}|(Z^*Z)_{ij}|^2
&=\sum_{i\in[n]}\sum_{j\in[n]}\sum_{k\in[d]}\sum_{\ell\in[d]}\overline{Z}_{ki}Z_{kj}Z_{\ell i}\overline{Z}_{\ell j},
\end{align*}
and formally differentiate
\begin{align*}
\frac{\partial\operatorname{EP}}{\partial \overline{Z}_{pq}}(Z)
&=\sum_{i\in[n]}\sum_{j\in[n]}\sum_{k\in[d]}\sum_{\ell\in[d]}(\delta_{pk}\delta_{qi}Z_{kj}Z_{\ell i}\overline{Z}_{\ell j}+\overline{Z}_{ki}Z_{kj}Z_{\ell i}\delta_{p\ell}\delta_{qj})\\
&=\sum_{j\in[n]}\sum_{\ell\in[d]}Z_{pj}Z_{\ell q}\overline{Z}_{\ell j}
+\sum_{i\in[n]}\sum_{k\in[d]}\overline{Z}_{ki}Z_{kq}Z_{p i}\\
&=\sum_{j\in[n]}\sum_{\ell\in[d]}Z_{pj}(Z^*)_{j\ell}Z_{\ell q}
+\sum_{i\in[n]}\sum_{k\in[d]}Z_{p i}(Z^*)_{ik}Z_{kq}
=(2ZZ^*Z)_{pq},
\end{align*}
and then double the Wirtinger gradient to obtain the gradient.
For $h_j$, we similarly expand and formally differentiate
\[
h_j(Z)
=\sum_{i\in[d]}Z_{ij}\overline{Z}_{ij}-1,
\quad
\frac{\partial h_j}{\partial \overline{Z}_{pq}}(Z)
=\sum_{i\in[d]}Z_{ij}\delta_{pi}\delta_{qj}
=Z_{pj}\delta_{qj}
=(z_je_j^*)_{pq},
\]
and then double the Wirtinger gradient to obtain the gradient.
\smartqed
\end{proof}

\begin{proposition}
\label{prop.focp}
For each $d,n\in\mathbb{N}$, the first-order critical points of the frame potential $\operatorname{FP}\colon\operatorname{S}(d,n)\to\mathbb{R}$ are precisely the matrices $Z\in\operatorname{S}(d,n)$ for which every column of $Z$ is an eigenvector of $ZZ^*$.
\end{proposition}

\begin{proof}
For every $Z\in\operatorname{S}(d,n)$, it holds that $\{\nabla h_j(Z)\}_{j\in[n]}$ are nonzero matrices with disjoint support, meaning they are linearly independent.
Thus, $Z$ is a first-order critical point if and only if $\nabla\operatorname{EP}(Z)$ is in the span of $\{\nabla h_j(Z)\}_{j\in[n]}$, if and only if for every $j\in[n]$, the $j$th column of $\nabla\operatorname{EP}(Z)$ is a scalar multiple of the $j$th column of $\nabla h_j(Z)$.
The result follows.
\end{proof}

Now that we have characterized the first-order critical points of the frame potential, we isolate a useful feature of the non-minimizing critical points.
To enunciate this feature, we need another definition:
We say $Z\in\operatorname{S}(d,n)$ is a \textbf{unit norm tight frame for a subspace} $S\subseteq\mathbb{C}^{d}$ if $ZZ^*$ is a multiple of the orthogonal projection matrix with image $S$.
In particular, the multiple equals $n/\operatorname{dim}(S)$ by a trace argument.

\begin{corollary}
\label{cor.focp feature}
Fix $d,n\in\mathbb{N}$, and suppose $Z\in\operatorname{S}(d,n)$ is a first-order critical point of the frame potential $\operatorname{FP}\colon\operatorname{S}(d,n)\to\mathbb{R}$ that is not a global minimizer.
Then there exist $d',n'\in\mathbb{N}$ satisfying
\[
d'<d,
\qquad
\frac{n'}{d'}>\max\bigg\{\frac{n}{d},1\bigg\}
\]
such that the maximum eigenvalue of $ZZ^*$ has a $d'$-dimensional eigenspace $E$ that contains exactly $n'$ columns of $Z$, which in turn form a unit norm tight frame for $E$.
\end{corollary}

\begin{proof}
Let $\Lambda$ denote the set of eigenvalues of $ZZ^*$, and let $d'$ denote the multiplicity of $\max\Lambda$ as an eigenvalue of $ZZ^*$.
Since $Z$ is not a unit norm tight frame by assumption, it holds that $d'<d$.
Next, let $P_\lambda$ denote the orthogonal projection matrix onto the eigenspace of $ZZ^*$ corresponding to $\lambda\in\Lambda$.
Then the spectral theorem gives
\[
\sum_{\lambda\in\Lambda}\lambda \frac{\operatorname{tr}P_\lambda}{d}
=\frac{1}{d}\operatorname{tr}\bigg(\sum_{\lambda\in\Lambda}\lambda P_\lambda\bigg)
=\frac{1}{d}\operatorname{tr}(ZZ^*)
=\frac{1}{d}\operatorname{tr}(Z^*Z)
=\frac{n}{d},
\]
i.e., $\frac{n}{d}$ is a weighted average of the eigenvalues of $ZZ^*$.
Since $Z$ is not a unit norm tight frame by assumption, there are at least two distinct eigenvalues, and so $\max\Lambda>\frac{n}{d}$.
In the case where $n<d$, the analogous weighted average of eigenvalues of $Z^*Z$ equals $1$, and so $\max\Lambda>1$ follows from the assumption that the columns of $Z$ are not orthonormal.
Overall, we have
\[
d'<d,
\qquad
\max\Lambda>\max\bigg\{\frac{n}{d},1\bigg\}.
\]
Next, consider the assignment $f\colon[n]\to\Lambda$ afforded by Proposition~\ref{prop.focp}.
Then the spectral theorem gives
\[
\sum_{\lambda\in\Lambda}
\sum_{j\in f^{-1}(\lambda)}z_jz_j^*
=\sum_{j\in[n]}z_jz_j^*
=ZZ^*
=\sum_{\lambda\in\Lambda}\lambda P_\lambda.
\]
Since distinct eigenspaces are orthogonal, then for each $\lambda'\in\Lambda$, we have
\[
\sum_{j\in f^{-1}(\lambda')}z_jz_j^*
=P_{\lambda'}\sum_{\lambda\in\Lambda}
\sum_{j\in f^{-1}(\lambda)}z_jz_j^*
=P_{\lambda'}\sum_{\lambda\in\Lambda}\lambda P_\lambda
=\lambda' P_{\lambda'},
\]
i.e., $\{z_jz_j^*\}_{j\in f^{-1}(\lambda')}$ is a unit norm tight frame for the eigenspace corresponding to $\lambda'$.
Taking $n':=|f^{-1}(\max\Lambda)|$ gives $\max\Lambda=\frac{n'}{d'}$, and the result follows.
\smartqed
\end{proof}

Next, we will characterize the second-order critical points of the frame potential.
To accomplish this, we first recall how Wirtinger calculus captures Hessian information.
The Wirtinger derivative $\frac{\partial f}{\partial z_i}\colon\mathbb{C}^n\to\mathbb{C}$ is defined by
\[
\frac{\partial f}{\partial z_i}
:=\frac{1}{2}\bigg(\frac{\partial(f\circ g)}{\partial x_i}-\mathrm{i}\frac{\partial(f\circ g)}{\partial y_i}\bigg)\circ g^{-1},
\]
and the \textbf{Wirtinger Hessians} $\frac{\partial^2f}{\partial z^2}, \frac{\partial^2f}{\partial z\partial \overline{z}}, \frac{\partial^2f}{\partial \overline{z}\partial z}, \frac{\partial^2f}{\partial \overline{z}^2}\colon\mathbb{C}^n\to\mathbb{C}^{n\times n}$ are defined by
\[
\bigg(\frac{\partial^2f}{\partial z\partial \overline{z}}(u)\bigg)_{ij}
:=\frac{\partial}{\partial z_i}\bigg(\frac{\partial f}{\partial \overline{z}_j}\bigg)(u),
\]
and similarly for the other Wirtinger Hessians.
The following result is essentially contained in the monograph~\cite{Kreutz:09}, but the proof is straightforward, so we provide a sketch for completeness:

\begin{proposition}
\label{prop.wirtinger socp}
Define $g\colon\mathbb{R}^n\times\mathbb{R}^n\to\mathbb{C}^n$ by $g(x,y):=x+\mathrm{i}y$ and $h\colon\mathbb{C}^n\to\mathbb{C}^n\times\mathbb{C}^n$ by $h(z):=(z,\overline{z})$.
Consider any $f\colon\mathbb{C}^n\to\mathbb{R}$ such that $f\circ g$ is twice continuously differentiable.
For every $u,v\in\mathbb{R}^n\times\mathbb{R}^n$, it holds that
\[
v^\top\nabla^2(f\circ g)(u)v
=(h\circ g)(v)^\top\left[\begin{array}{cc}
\frac{\partial^2f}{\partial z^2}(g(u))&\frac{\partial^2f}{\partial z\partial\overline{z}}(g(u))\\
\frac{\partial^2f}{\partial\overline{z}\partial z}(g(u))&\frac{\partial^2f}{\partial \overline{z}^2}(g(u))
\end{array}\right](h\circ g)(v).
\]
\end{proposition}

\begin{proof}[sketch]
By a standard abuse of notation, we write
\[
\frac{\partial}{\partial x_i}=\frac{\partial}{\partial z_i}+\frac{\partial}{\partial \overline{z}_i},
\qquad
\frac{\partial}{\partial y_i}=\mathrm{i}\bigg(\frac{\partial}{\partial z_i}-\frac{\partial}{\partial \overline{z}_i}\bigg).
\]
Then, for example,
\[
\frac{\partial^2}{\partial y_i\partial x_j}
=\mathrm{i}\bigg(\frac{\partial}{\partial z_i}-\frac{\partial}{\partial \overline{z}_i}\bigg)\bigg(\frac{\partial}{\partial z_j}+\frac{\partial}{\partial \overline{z}_j}\bigg),
\]
and expanding implies
\[
\nabla_{yx}^2=\mathrm{i}\bigg(\frac{\partial^2}{\partial z^2}+\frac{\partial^2}{\partial z\partial\overline{z}}-\frac{\partial^2}{\partial\overline{z}\partial z}-\frac{\partial^2}{\partial \overline{z}^2}\bigg).
\]
The same approach delivers similar expressions for $\nabla_{xx}^2$, $\nabla_{xy}^2$, and $\nabla_{yy}^2$.
The desired result can then be verified by expanding the left-hand side in terms of these expressions and rearranging to obtain the right-hand side.
\smartqed
\end{proof}

Proposition~\ref{prop.wirtinger socp} will help us identify the second-order critical points of the frame potential.
To this end, the following lemma summarizes a straightforward calculation:

\begin{lemma}
\label{lem.wirtinger hessians}
For $\operatorname{EP},h_1,\ldots,h_n\colon\mathbb{C}^{d\times n}\to\mathbb{R}$ defined in Lemma~\ref{lem.euclidean gradient}, it holds that
\[
\frac{\partial^2\operatorname{EP}}{\partial Z_{pq}\partial Z_{rs}}(Z)
=2\overline{Z}_{rq}\overline{Z}_{ps},
\qquad
\frac{\partial^2\operatorname{EP}}{\partial \overline{Z}_{pq}\partial \overline{Z}_{rs}}(Z)
=2Z_{rq}Z_{ps},
\]
\[
\frac{\partial^2\operatorname{EP}}{\partial Z_{pq}\partial \overline{Z}_{rs}}(Z)
=2\bigg(\delta_{pr}(Z^*Z)_{qs}+\delta_{qs}(ZZ^*)_{rp}\bigg),
\]
\[
\frac{\partial^2 h_j}{\partial Z_{pq}\partial Z_{rs}}(Z)
=\frac{\partial^2 h_j}{\partial \overline{Z}_{pq}\partial \overline{Z}_{rs}}(Z)
=0,
\qquad
\frac{\partial^2 h_j}{\partial Z_{pq}\partial \overline{Z}_{rs}}(Z)
=\delta_{pr}\delta_{qj}\delta_{sj}.
\]
\end{lemma}

We may now prove the main result of this section:

\begin{proof}[of Theorem~\ref{thm.socp}]
One containment is trivial: every global minimizer is necessarily a local minimizer, which is necessarily a second-order critical point.
For the other direction, we show that every point $Z\in\operatorname{S}(d,n)$ that is not a global minimizer is necessarily not a second-order critical point.
Since every second-order critical point is necessarily a first-order critical point, the implication trivially holds if $Z$ is not a first-order critical point.
Suppose $Z=[z_1\cdots z_n]$ is a first-order critical point that is not a global minimizer.
The corresponding Lagrange multipliers $\{\lambda_j\}_{j\in[n]}$ are the unique solution to
\[
ZZ^*z_j
=-\frac{1}{2}\lambda_jz_j
\]
for $j\in[n]$.
Consider the map $L\colon\mathbb{C}^{d\times n}\to\mathbb{R}$ defined by
\[
L(Z)
:=\operatorname{EP}(Z)+\sum_{j\in[n]}\lambda_jh_j(Z).
\]
To show that $Z$ is not a second-order critical point, Proposition~\ref{prop.wirtinger socp} gives that it suffices to find $Y\in\mathbb{C}^{d\times n}$ such that
\[
\operatorname{Re}\operatorname{tr}(\nabla h_j(Z)^*Y)
=0
\]
for every $j\in[n]$ and
\begin{align*}
&\sum_{p\in[d]}\sum_{q\in[n]}\sum_{r\in[d]}\sum_{s\in[n]}
\bigg(
\frac{\partial^2L(Z)}{\partial Z_{pq}\partial Z_{rs}}Y_{pq}Y_{rs}+
\frac{\partial^2L(Z)}{\partial Z_{pq}\partial \overline{Z}_{rs}}Y_{pq}\overline{Y}_{rs}\\
&\qquad\qquad\qquad\qquad\qquad+
\frac{\partial^2L(Z)}{\partial \overline{Z}_{pq}\partial Z_{rs}}\overline{Y}_{pq}Y_{rs}+
\frac{\partial^2L(Z)}{\partial \overline{Z}_{pq}\partial \overline{Z}_{rs}}\overline{Y}_{pq}\overline{Y}_{rs}\bigg)
<0.
\end{align*}
To enunciate our choice for $Y$, let $J\subseteq[n]$ denote the index set of the $n'$ columns of $Z$ implicated by Corollary~\ref{cor.focp feature}.
Then Corollary~\ref{cor.focp feature} gives that $\{z_j\}_{j\in J}$ spans a $d'$-dimensional eigenspace $E$ of $ZZ^*$ with $d'<\min\{d,n'\}$.
Since $d'<d$, there is another eigenspace $E'$ of $ZZ^*$, and we select $u\in E'$ with $\|u\|^2=1$.
Since $d'<n'$, then $\{z_j\}_{j\in J}$ is linearly dependent, and we select $v\in\mathbb{C}^n$ with support contained in $J$ such that $Zv=0$ and $\|v\|^2=1$.
We take $Y:=uv^*$, and it remains to verify that $Y$ satisfies the above conditions.
First, Lemma~\ref{lem.euclidean gradient} gives
\[
\operatorname{Re}\operatorname{tr}(\nabla h_j(Z)^*Y)
=\operatorname{Re}\operatorname{tr}(2e_j z_j^*uv^*)
=2\operatorname{Re}(v^*e_j z_j^*u)
=0,
\]
where the last step follows from the facts that $v^*e_j=0$ whenever $j\not\in J$ and $z_j^*u=0$ whenever $j\in J$.
Next, Lemma~\ref{lem.wirtinger hessians} gives
\begin{align*}
\sum_{p\in[d]}\sum_{q\in[n]}\sum_{r\in[d]}\sum_{s\in[n]}
\frac{\partial^2\operatorname{EP}(Z)}{\partial Z_{pq}\partial Z_{rs}}Y_{pq}Y_{rs}
&=\sum_{p\in[d]}\sum_{q\in[n]}\sum_{r\in[d]}\sum_{s\in[n]}
2\overline{Z}_{rq}\overline{Z}_{ps}u_p\overline{v}_qu_r\overline{v}_s\\
&=2(v^*Z^*u)^2
=0,
\end{align*}
where the last step follows from the fact that $v$ and $Z^*u$ have disjoint support.
We similarly have
\[
\sum_{p\in[d]}\sum_{q\in[n]}\sum_{r\in[d]}\sum_{s\in[n]}
\frac{\partial^2\operatorname{EP}(Z)}{\partial \overline{Z}_{pq}\partial \overline{Z}_{rs}}\overline{Y}_{pq}\overline{Y}_{rs}
=0.
\]
Considering Lemma~\ref{lem.wirtinger hessians}, it remains to show that
\begin{equation}
\label{eq.final hessian inequality}
\operatorname{Re}\bigg(\sum_{p\in[d]}\sum_{q\in[n]}\sum_{r\in[d]}\sum_{s\in[n]}
\frac{\partial^2L(Z)}{\partial Z_{pq}\partial \overline{Z}_{rs}}Y_{pq}\overline{Y}_{rs}\bigg)
<0.
\end{equation}
To this end, let $\alpha$ and $\beta$ denote the eigenvalues corresponding to eigenspaces $E$ and $E'$, respectively.
Notably, Corollary~\ref{cor.focp feature} gives $\alpha>\beta$, and this will be the source of our desired inequality.
Lemma~\ref{lem.wirtinger hessians} gives
\begin{align}
\nonumber
&\sum_{p\in[d]}\sum_{q\in[n]}\sum_{r\in[d]}\sum_{s\in[n]}
\frac{\partial^2\operatorname{EP}(Z)}{\partial Z_{pq}\partial \overline{Z}_{rs}}Y_{pq}\overline{Y}_{rs}\\
\nonumber
&=\sum_{p\in[d]}\sum_{q\in[n]}\sum_{r\in[d]}\sum_{s\in[n]}
2\bigg(\delta_{pr}(Z^*Z)_{qs}+\delta_{qs}(ZZ^*)_{rp}\bigg) u_p\overline{v}_q \overline{u}_r v_s\\
\nonumber
&=2\sum_{q\in[n]}\sum_{r\in[d]}\sum_{s\in[n]}(Z^*Z)_{qs}u_r\overline{v}_q\overline{u}_rv_s
+2\sum_{p\in[d]}\sum_{r\in[d]}\sum_{s\in[n]}(ZZ^*)_{rp}u_p\overline{v}_s\overline{u}_rv_s\\
\label{eq.ep hessian}
&=2\|Zv\|^2\|u\|^2+2\|Z^*u\|^2\|v\|^2
=2\beta,
\end{align}
where the last step uses the facts that $Zv=0$, $\|v\|^2=1$, and $\|Z^*u\|^2=u^*ZZ^*u=\beta$.
Next, Lemma~\ref{lem.wirtinger hessians} gives
\begin{align}
\sum_{p\in[d]}\sum_{q\in[n]}\sum_{r\in[d]}\sum_{s\in[n]}
\frac{\partial^2 h_j(Z)}{\partial Z_{pq}\partial \overline{Z}_{rs}}Y_{pq}\overline{Y}_{rs}
\nonumber
&=\sum_{p\in[d]}\sum_{q\in[n]}\sum_{r\in[d]}\sum_{s\in[n]}
\delta_{pr}\delta_{qj}\delta_{sj} u_p\overline{v}_q \overline{u}_rv_s\\
\label{eq.hj hessian}
&=\|u\|^2|v_j|^2
=|v_j|^2.
\end{align}
We combine \eqref{eq.ep hessian} and \eqref{eq.hj hessian} to get
\[
\sum_{p\in[d]}\sum_{q\in[n]}\sum_{r\in[d]}\sum_{s\in[n]}
\frac{\partial^2L(Z)}{\partial Z_{pq}\partial \overline{Z}_{rs}}Y_{pq}\overline{Y}_{rs}
=2\beta+\sum_{j\in[n]}\lambda_j|v_j|^2
=2(\beta-\alpha)
<0,
\]
where the last equality follows from the facts that $\lambda_j=-2\alpha$ whenever $j\in J$, $v_j=0$ whenever $j\not\in J$, and $\|v\|^2=1$.
This implies the desired inequality \eqref{eq.final hessian inequality}.
\smartqed
\end{proof}

\section{Convexity and topology}
\label{sec.topology}

Given a self-adjoint matrix $A\in\mathbb{C}^{d\times d}$, let $\lambda(A)\in\mathbb{R}^d$ denote the vector of eigenvalues of $A$ sorted in monotonically decreasing order.
This section proves the following strengthening of the Benedetto--Fickus theorem:

\begin{theorem}
\label{thm.generalized frame potential}
Given a convex function $f\colon\mathbb{R}^d\to\mathbb{R}$, consider $\operatorname{FP}_f\colon\operatorname{S}(d,n)\to\mathbb{R}$ defined by $\operatorname{FP}_f(Z):=f(\lambda(ZZ^*))$.
Then $\operatorname{FP}_f$ has no spurious local minimizers.
\end{theorem}

Notably, if $f(\cdot)=\|\cdot\|^2$, then $\operatorname{FP}_f$ is the frame potential.
We will prove Theorem~\ref{thm.generalized frame potential} using ideas from both convexity and topology.
Recall that given topological spaces $(X,\tau_X)$ and $(Y,\tau_Y)$, we say $f\colon X\to Y$ is \textbf{open at $x\in X$} if for every $U\in \tau_X$ such that $x\in U$, there exists $V\in\tau_Y$ such that $f(x)\in V\subseteq f(U)$.
Furthermore, $f$ is \textbf{open} if for every $x\in X$, $f$ is open at $x$.
Throughout this section, $\operatorname{int}(U)$ and $\operatorname{cl}(U)$ denote the topological interior and closure of $U$, respectively.
The following observation appears to originate with the recent preprint~\cite{NouiehedR:18}:

\begin{lemma}
\label{lem.convex of open}
Given a topological space $X$, a convex subset $Y$ of a real topological vector space, an open map $g\colon X\to Y$ in the subspace topology for $Y$, and a convex function $f\colon Y\to\mathbb{R}$, it holds that $f\circ g$ has no spurious local minimizers.
\end{lemma}

\begin{proof}
Suppose $x_0\in X$ is a local minimizer of $f\circ g$, i.e., there exists $U\in\tau_X$ with $x_0\in U$ such that $(f\circ g)(x)\geq(f\circ g)(x_0)$ for every $x\in U$.
Since $g$ is open, it follows that $V:=g(U)\in\tau_Y$ contains $g(x_0)$ and has the property that $f(y)\geq f(g(x_0))$ for every $y\in V$.
That is, $g(x_0)$ is a local minimizer of $f$.
Since $f$ is convex, $g(x_0)$ is necessarily a global minimizer of $f$, meaning $f(y)\geq f(g(x_0))$ for every $y\in Y$.
Then $f(g(x))\geq f(g(x_0))$ for every $x\in X$, i.e., $x_0$ is a global minimizer of $f\circ g$.
\end{proof}

As such, to prove Theorem~\ref{thm.generalized frame potential}, it suffices to factor $\operatorname{FP}_f$ as the composition of an open map with a convex function.
To proceed, we need a definition.
Given $x,y\in\mathbb{R}^d$, we say $x$ \textbf{interlaces} $y$ and write $x\sqsupseteq y$ if
\[
x_1\geq y_1\geq x_2 \geq y_2\geq\cdots\geq x_d\geq y_d.
\]
Let $\operatorname{E}(d,n)$ denote the polytope of matrices $[\lambda_1\cdots\lambda_n]\in\mathbb{R}^{d\times n}$ such that
\[
\lambda_n\sqsupseteq\cdots\sqsupseteq\lambda_1\sqsupseteq0
\qquad
\text{and}
\qquad
\sum_{i\in[d]}(\lambda_j)_i=j
\quad
\forall j\in[n].
\]
For example, one may verify that
\[
\frac{1}{4}\left[
\begin{array}{ccccc}
4&6&7&8&9\\
0&2&4&6&7\\
0&0&1&2&4
\end{array}
\right]
\in\operatorname{E}(3,5).
\]
Throughout this section, we use $\Lambda$ to denote a member of $\operatorname{E}(d,n)$.

Define $\Sigma\colon\operatorname{S}(d,n)\to\operatorname{E}(d,n)$ so that the $j$th column of $\Sigma([z_1\cdots z_n])$ is the vector of eigenvalues of $\sum_{i\in[j]}z_iz_i^*$ sorted in monotonically decreasing order.
The fact that $\Sigma(Z)\in\operatorname{E}(d,n)$ for every $Z\in\operatorname{S}(d,n)$ essentially follows from Cauchy's interlacing theorem and is established in the proof of the ($\Rightarrow$) direction of Theorem~2 in~\cite{Cahill:13}. 
In words, $\Sigma(Z)$ is known as the \textbf{eigensteps} of $Z$.
(We use $\Sigma$ since it looks like an E for \textit{eigen} and sounds like an S for \textit{steps}.)
Observe the relationship between eigensteps and Gelfand--Tsetlin polytopes~\cite{DeLoeraM:04}.
Eigensteps were introduced in~\cite{Cahill:13} to explicitly construct the set of unit norm tight frames, and were later used to establish the path-connectivity of this set~\cite{CahillMS:17}.

Considering $\lambda(ZZ^*)=\Sigma(Z)e_d$, then $\operatorname{FP}_f$ is the composition of $Z\mapsto\Sigma(Z)$ with the convex function $\Lambda\mapsto f(\Lambda e_d)$.
Thus, we would like to show that the eigensteps map $\Sigma$ is open.
The following lemma summarizes our approach to this end:

\begin{lemma}
\label{lem.suff cond for open}
Given topological spaces $X$ and $Y$, consider a point $x\in X$ and a map $f\colon X\to Y$ that is continuous at $x$.
Suppose there exists $W\in \tau_Y$ such that $f(x)\in W$ and a continuous map $g\colon W\to X$ such that
\[
(g\circ f)(x)=x
\qquad
\text{and}
\qquad
f\circ g=\operatorname{id}_W.
\]
Then $f$ is open at $x$.
\end{lemma}

\begin{proof}
Given any $U\in\tau_X$ such that $x\in U$, take $V:=g^{-1}(U\cap f^{-1}(W))$.
Then $V\in\tau_Y$ since $W$ is open, $f(x)\in W$, $f$ is continuous at $x$, $U$ is open, and $g$ is continuous.
Next, since $x\in U$ and $f(x)\in W$, we have $x\in U\cap f^{-1}(W)$, and so $g(f(x))=(g\circ f)(x)=x$ gives $f(x)\in g^{-1}(U\cap f^{-1}(W))=V$.
It remains to verify that $V\subseteq f(U)$.
Select any $v\in V=g^{-1}(U\cap f^{-1}(W))\subseteq W$.
Then $v=(f\circ g)(v)=f(g(v))\in f(U\cap f^{-1}(W))\subseteq f(U)$.
\smartqed
\end{proof}

Due to the nature of the eigensteps map, our construction of $g$ uses a bit of matrix analysis.
We start with a couple of lemmas:

\begin{lemma}
\label{lem.multiplicity decrease}
Given a self-adjoint matrix $A\in\mathbb{C}^{d\times d}$ and eigenvalue $\lambda$ of multiplicity $m$, let $P_\lambda\in\mathbb{C}^{d\times d}$ denote orthogonal projection onto the corresponding eigenspace.
For $z\in\mathbb{C}^d$ with $P_\lambda z\neq0$, the multiplicity of $\lambda$ as an eigenvalue of $A+zz^*$ is $m-1$.
\end{lemma}

\begin{proof}
Let $\{v_i\}_{i\in[n]}$ denote any orthonormal basis of eigenvectors of $A$ with corresponding eigenvalues $\{\lambda_i\}_{i\in[n]}$.
We apply the matrix determinant lemma to express the characteristic polynomial of $A+zz^*$:
\begin{align*}
p_{A+zz^*}(t)
=\operatorname{det}((tI-A)-zz^*)
&=\operatorname{det}(tI-A)\cdot (1-z^*(tI-A)^{-1}z)\\
&=\bigg(\prod_{i\in [n]}(t-\lambda_i)\bigg)\bigg(1-\sum_{j\in[n]}\frac{|\langle v_j,z\rangle|^2}{t-\lambda_j}\bigg)\\
&=\prod_{i\in[n]}(t-\lambda_i)-\sum_{j\in[n]}|\langle v_j,z\rangle|^2\prod_{\substack{i\in[n]\\i\neq j}}(t-\lambda_j).
\end{align*}
We rewrite this expression in terms of the multiplicity $m(\lambda')$ of each eigenvalue $\lambda'$ of $A$ to find polynomials $q(t)$ and $r(t)$ with $q(\lambda)\neq0\neq r(\lambda)$ such that
\begin{align*}
p_{A+zz^*}(t)
&=\prod_{\lambda'}(t-\lambda')^{m(\lambda')}-\sum_{\lambda'}\|P_{\lambda'}z\|^2(t-\lambda')^{m(\lambda')-1}\prod_{\lambda''\neq\lambda'}(t-\lambda'')^{m(\lambda'')}\\
&=(t-\lambda)^mq(t)-\|P_\lambda z\|^2(t-\lambda)^{m-1}r(t).
\end{align*}
Since $\|P_\lambda z\|^2\neq0$ by assumption, the result follows.
\smartqed
\end{proof}

\begin{lemma}
\label{lem.open and dense in S}
The preimage $\Sigma^{-1}(\operatorname{int}(\operatorname{E}(d,n)))$ is dense in $\operatorname{S}(d,n)$.
\end{lemma}

\begin{proof}
Given arbitrary $Z=[z_1\cdots z_n]\in\operatorname{S}(d,n)$ and $\epsilon>0$, we identify a matrix $\tilde{Z}=[\tilde{z}_1\cdots\tilde{z}_n]\in\Sigma^{-1}(\operatorname{int}(\operatorname{E}(d,n)))$ such that $\|\tilde{Z}-Z\|_F<\epsilon$.
Put $\tilde{z}_1:=z_1$, and then for each $j\in[n-1]$, select any unit vector $\tilde{z}_{j+1}$ in the open $2$-ball of radius $\epsilon/n$ about $z_{j+1}$ that avoids the orthogonal complements of the eigenspaces of $\sum_{i\in[j]}\tilde{z}_i\tilde{z}_i^*$.
Then $\Sigma(\tilde{Z})\in\operatorname{int}(\operatorname{E}(d,n))$ by Lemma~\ref{lem.multiplicity decrease} and $\|\tilde{Z}-Z\|_F<\epsilon$ by the triangle inequality.
\smartqed
\end{proof}

Let $\mathbb{T}^k\times\operatorname{U}(d-k)$ denote the group of matrices in $\operatorname{U}(d)$ of the form
\[
\left[\begin{array}{cc}
D&0\\
0&U
\end{array}\right]
\]
with diagonal $D\in\operatorname{U}(k)$, and consider the following group of $n$-tuples of matrices:
\[
\operatorname{U}(d,n)
:=\operatorname{U}(d)
\times (\mathbb{T}^1\times\operatorname{U}(d-1))
\times \cdots
\times (\mathbb{T}^{d-1}\times\operatorname{U}(1))
\times (\mathbb{T}^d)^{n-d}.
\]
(We use $\operatorname{U}$ since each component of the tuple is \textit{unitary}.)
The following lemma describes how this set allows us to map back from $\operatorname{E}(d,n)$ to $\operatorname{S}(d,n)$.
This portion of our proof appears to break in the real case since the orthogonal group $\operatorname{O}(1)$ is discrete; it would be interesting if the real case could be treated differently.

\begin{lemma}
\label{lem.partial inverse}
There exists a continuous map $\Phi\colon\operatorname{E}(d,n)\times\operatorname{U}(d,n)\to\operatorname{S}(d,n)$ such that
\begin{itemize}
\item[(a)]
for every $Z\in\operatorname{S}(d,n)$, there exists $U\in\operatorname{U}(d,n)$ such that $\Phi(\Sigma(Z),U)=Z$, and
\item[(b)]
for every $\Lambda\in\operatorname{E}(d,n)$ and $U\in\operatorname{U}(d,n)$, it holds that $\Sigma(\Phi(\Lambda,U))=\Lambda$.
\end{itemize}
\end{lemma}

\begin{proof}
We will define $\Phi$ explicitly on $\operatorname{int}(\operatorname{E}(d,n))\times\operatorname{U}(d,n)$ and then extend continuously to the full domain.
For each $\Lambda=[\lambda_1\cdots\lambda_n]\in\operatorname{int}(\operatorname{E}(d,n))$ and $j\in[n]$, denote $r_j:=\min\{j+1,d\}$, define $v_j,w_j\in\mathbb{R}^{r_j}$ by
\[
(v_j)_i:=\sqrt{-\frac{\displaystyle\prod_{i'\in[r_j]}((\lambda_j)_i-(\lambda_{j+1})_{i'})}{\displaystyle\prod_{\substack{i'\in[r_j]\\i'\neq i}}((\lambda_j)_i-(\lambda_{j})_{i'})}},
\quad
(w_j)_i:=\sqrt{\frac{\displaystyle\prod_{i'\in[r_j]}((\lambda_{j+1})_i-(\lambda_{j})_{i'})}{\displaystyle\prod_{\substack{i'\in[r_j]\\i'\neq i}}((\lambda_{j+1})_i-(\lambda_{j+1})_{i'})}},
\]
and define $W_j\in\operatorname{U}(r_j)$ by
\[
(W_j)_{i,i'}:=\frac{(v_j)_i (w_j)_{i'}}{(\lambda_{j+1})_{i'}-(\lambda_j)_i}.
\]
By Theorem~7 in~\cite{Cahill:13}, $\Sigma^{-1}(\Lambda)$ equals the set of matrices $[z_1\cdots z_n]\in\mathbb{C}^{d\times n}$ obtained from all $(U_1,V_1,\ldots,V_{n-1})\in\operatorname{U}(d,n)$ by taking $z_1:=U_1e_1$ and applying the following iteration for $j\in[n-1]$:
\[
z_{j+1}:=U_jV_j\left[\begin{array}{c}v_j\\0\end{array}\right],
\qquad
U_{j+1}:=U_jV_j\left[\begin{array}{cc}W_j&0\\0&I_{d-r_j}\end{array}\right].
\]
This process defines a map $\Phi_0\colon\operatorname{int}(\operatorname{E}(d,n))\times\operatorname{U}(d,n)\to\operatorname{S}(d,n)$ such that
\begin{itemize}[leftmargin=*,labelindent=15pt,labelsep=5pt]
\item[(a')]
for every $Z\in\Sigma^{-1}(\operatorname{int}(\operatorname{E}(d,n)))$, there exists $U\in\operatorname{U}(d,n)$ such that $\Phi_0(\Sigma(Z),U)=Z$, and
\item[(b')]
for every $\Lambda\in\operatorname{int}(\operatorname{E}(d,n))$ and $U\in\operatorname{U}(d,n)$, it holds that $\Sigma(\Phi_0(\Lambda,U))=\Lambda$.
\end{itemize}
Observe that $R\colon\Lambda\mapsto(\{(v_j)_i^2\}_{i,j},\{(W_j)_{i,i'}^2\}_{j,i,i'})$ is a rational map with bounded domain, and furthermore, $R$ is bounded since $\Phi_0$ is bounded.
It follows that $R$ is Lipschitz, and so $\Phi_0$ is uniformly continuous since $\Phi_0(\Lambda,U)$ is a polynomial in $(\sqrt{R(\Lambda)},U)$ with bounded domain.
In addition, $\operatorname{S}(d,n)$ is a complete metric space, and so there exists a unique continuous extension $\Phi$ of $\Phi_0$ to all of $\operatorname{E}(d,n)\times\operatorname{U}(d,n)$.

We now verify properties (a) and (b).
For (a), take any $Z\in\operatorname{S}(d,n)$.
By Lemma~\ref{lem.open and dense in S}, there exists a sequence $\{Z_i\}_{i=1}^\infty$ in $\operatorname{S}(d,n)$ such that $\Sigma(Z_i)\in\operatorname{int}(\operatorname{E}(d,n))$ for every $i$ and $Z_i\to Z$.
By (a'), it holds that for each $i$, there exists $U_i\in\operatorname{U}(d,n)$ such that $\Phi(\Sigma(Z_i),U_i)=Z_i$.
By compactness, there is a subsequence of $\{U_i\}_{i=1}^\infty$ with some limit point $U\in\operatorname{U}(d,n)$.
By continuity, it follows that $\Phi(\Sigma(Z),U)=Z$.
For (b), take any $\Lambda\in\operatorname{E}(d,n)$ and $U\in\operatorname{U}(d,n)$.
There exists a sequence $\{\Lambda_i\}_{i=1}^\infty$ in $\operatorname{int}(\operatorname{E}(d,n))$ such that $\Lambda_i\to\Lambda$.
By (b'), it holds that $\Sigma(\Phi(\Lambda_i,U))=\Lambda_i$ for each $i$.
Since $\Sigma$ and $\Phi$ are continuous, taking limits of both sides gives $\Sigma(\Phi(\Lambda,U))=\Lambda$.
\end{proof}

\begin{lemma}
\label{lem.eigensteps is open}
The eigensteps map $\Sigma\colon\operatorname{S}(d,n)\to\operatorname{E}(d,n)$ is open in the standard subspace topologies for $\operatorname{S}(d,n)\subseteq\mathbb{C}^{d\times n}$ and $\operatorname{E}(d,n)\subseteq\mathbb{R}^{d\times n}$.
\end{lemma}

\begin{proof}
Take any $Z\in\operatorname{S}(d,n)$.
By Lemma~\ref{lem.partial inverse}(a), there exists $U\in\operatorname{U}(d,n)$ such that $\Phi(\Sigma(Z),U)=Z$.
We may apply Lemma~\ref{lem.suff cond for open} with
\[
f:=\Sigma,
\qquad
W:=\operatorname{E}(d,n),
\qquad
g\colon\Lambda\mapsto\Phi(\Lambda,U).
\]
Indeed, Lemma~\ref{lem.partial inverse}(a) gives $(g\circ f)(Z)=\Phi(\Sigma(Z),U)=Z$ and Lemma~\ref{lem.partial inverse}(b) gives $(f\circ g)(\Lambda)=\Sigma(\Phi(\Lambda,U))=\Lambda$ for every $\Lambda\in\operatorname{E}(d,n)$.
Thus, $\Sigma$ is open at $Z$.
Since $Z$ was arbitrary, it follows that $\Sigma$ is open.
\smartqed
\end{proof}

We may now prove the main result of this section:

\begin{proof}[of Theorem~\ref{thm.generalized frame potential}]
By Lemma~\ref{lem.eigensteps is open}, the eigensteps map $\Sigma$ is open.
As such, $\operatorname{FP}_f$ is the composition of the open map $Z\mapsto \Sigma(Z)$ with the convex map $\Lambda\mapsto f(\Lambda e_d)$, and so the result follows from Lemma~\ref{lem.convex of open}.
\smartqed
\end{proof}

\section{Geometric invariant theory}
\label{sec.git}

The recent literature showcases an industry of using geometric methods to prove fundamental results in frame theory~\cite{NeedhamS:21,NeedhamS:arXiv1,NeedhamS:arXiv2}.
This section represents another step in this direction.
We say $Z\in\mathbb{C}^{d\times n}$ is \textbf{full spark} if every $d\times d$ submatrix of $Z$ is invertible~\cite{AlexeevCM:12}.
We will prove the following strengthening of the Benedetto--Fickus theorem in the special case where $n>d$:

\begin{theorem}
\label{thm.gradient flow to untf}
Fix $d,n\in\mathbb{N}$ with $n>d$, and consider the gradient flow $F\colon\operatorname{S}(d,n)\times[0,\infty)\to\operatorname{S}(d,n)$ defined by the differential equation
\[
F(Z_0,0)=Z_0,
\qquad
\frac{d}{dt}F(Z_0,t)=-\operatorname{grad}\operatorname{FP}(F(Z_0,t)),
\]
where $\operatorname{grad}$ denotes Riemannian gradient.
If $Z_0$ is full spark, then $\lim_{t\to\infty}F(Z_0,t)$ is a unit norm tight frame.
\end{theorem}

Given a critical point that is not a unit norm tight frame, almost every member of any neighborhood of the critical point is full spark, which by Theorem~\ref{thm.gradient flow to untf} is sent by gradient flow to a unit norm tight frame.
The absence of a basin of attraction indicates that the critical point is not a local minimizer.
Since the frame potential is real analytic (in fact, polynomial), this can be made rigorous with a standard application of the \L{}ojasiewicz inequality (\`{a} la~\cite{AbsilK:06}), thereby proving the Benedetto--Fickus theorem.
Furthermore, Theorem~\ref{thm.gradient flow to untf} easily extends to the real case since the real submanifold of $\operatorname{S}(d,n)$ is invariant under gradient flow.
To prove Theorem~\ref{thm.gradient flow to untf}, we identify a property P such that
\begin{itemize}
\item
every full spark member of $\operatorname{S}(d,n)$ satisfies P, and
\item
gradient flow (and its limit) preserves P, but
\item
no non-minimizing critical point of the frame potential satisfies P.
\end{itemize}
It turns out that this property originated from the \textit{geometric invariant theory} developed by Mumford~\cite{MumfordFK:94} using ideas from Hilbert~\cite{Hilbert:93}; see~\cite{Thomas:06} for a survey.
We start by introducing a few terms in some level of generality.

Let $G$ be a reductive group acting linearly on a finite-dimensional complex vector space $V$, such as $\operatorname{GL}(V)$, $\operatorname{SL}(V)$, or $\operatorname{O}(V)$.
A nonzero vector $v\in V$ is said to be \textbf{unstable} under the action of $G$ if the closure $\operatorname{cl}(G\cdot v)$ of the orbit of $v$ contains $0$, and otherwise $v$ is \textbf{semi-stable} under the action of $G$.

\begin{proposition}
\label{prop.semistable invariant}
Given a nonzero vector $v\in V$ that is semi-stable under the action of $G$, it holds that every member of $\operatorname{cl}(G\cdot v)$ is also semi-stable under the action of $G$.
\end{proposition}

\begin{proof}
We prove the contrapositive.
Suppose there exists a nonzero $w\in\operatorname{cl}(G\cdot v)$ that is unstable under the action of $G$.
Then there exist sequences $\{g_i\}_{i=1}^\infty$ and $\{h_j\}_{j=1}^\infty$ in $G$ such that $g_i\cdot v\to w$ and $h_j\cdot w\to0$.
For each $j$, the continuity of the linear action of $h_j$ on $V$ implies that
\[
h_j\cdot w
=h_j\cdot\lim_{i\to\infty}(g_i\cdot v)
=\lim_{i\to\infty}(h_j g_i\cdot v)
\in\operatorname{cl}(G\cdot v),
\]
and so $0=\lim_{j\to\infty}(h_j\cdot w)\in\operatorname{cl}(G\cdot v)$.
Thus, $v$ is unstable under the action of $G$.
\smartqed
\end{proof}

We are interested in the special case where $G=\operatorname{SL}(d)$ acts on $V=(\mathbb{C}^d)^{\otimes n}$ by 
\[
g\cdot (z_1\otimes\cdots\otimes z_n)
:=gz_1\otimes\cdots\otimes gz_n
\]
and extending linearly.
We pass to this space by ``tensoring up columns'' with $\tau\colon\mathbb{C}^{d\times n}\to(\mathbb{C}^d)^{\otimes n}$ defined by $[z_1\cdots z_n]\mapsto z_1\otimes\cdots\otimes z_n$.
Then $Z\in\operatorname{S}(d,n)$ satisfies the previously mentioned property P if $\tau(Z)$ is semi-stable under the action of $\operatorname{SL}(d)$.

\begin{lemma}
\label{lem.git separation}
Fix $d,n\in\mathbb{N}$ with $n>d$.
\begin{itemize}
\item[(a)]
For every full spark $Z\in\operatorname{S}(d,n)$, it holds that $\tau(Z)$ is semi-stable under the action of $\operatorname{SL}(d)$.
\item[(b)]
For every non-minimizing critical point $Z\in\operatorname{S}(d,n)$ of the frame potential, it holds that $\tau(Z)$ is unstable under the action of $\operatorname{SL}(d)$.
\end{itemize}
\end{lemma}

Assuming the lemma for now, we proceed with the proof of Theorem~\ref{thm.gradient flow to untf}.

\begin{proof}[sketch of Theorem~\ref{thm.gradient flow to untf}]
For every $Z\in\operatorname{S}(d,n)$, Lemma~\ref{lem.euclidean gradient} gives the Euclidean gradient $\nabla\operatorname{EP}(Z)=4ZZ^*Z$, and so the Riemannian gradient $\operatorname{grad}\operatorname{FP}(Z)$ is the matrix obtained by normalizing the columns of $ZZ^*Z$.
Consider the action of $\operatorname{GL}(d)\times(\mathbb{C}^\times)^n$ on $\mathbb{C}^{d\times n}$ defined by $(g,\{a_i\}_{i\in[n]})\cdot X:=gX\operatorname{diag}\{a_i\}_{i\in[n]}^{-1}$.
Then $\operatorname{grad}\operatorname{FP}(Z)$ is in the tangent space at $Z$ of the orbit $(\operatorname{GL}(d)\times(\mathbb{C}^\times)^n)\cdot Z$, and so $F(Z_0,t)$ resides in one such orbit for all $t\geq0$.
For every $g\in\operatorname{GL}(d)$, it holds that $g/\operatorname{det}(g)^{1/d}\in\operatorname{SL}(d)$, and so normalizing gives
\[
F(Z_0,t)\in(\operatorname{SL}(d)\times(\mathbb{C}^\times)^n)\cdot Z_0
\qquad
\forall t\geq0.
\]
Next, consider the action of $\operatorname{SL}(d)\times\mathbb{C}^\times$ on $(\mathbb{C}^d)^{\otimes n}$ defined by $(g,a)\cdot(z_1\otimes\cdots\otimes z_n):=a(gz_1\otimes\cdots\otimes gz_n)$ and extending linearly.
Then
\[
\tau(F(Z_0,t))\in(\operatorname{SL}(d)\times\mathbb{C}^\times)\cdot \tau(Z_0)
\qquad
\forall t\geq0.
\]
Suppose $Z_0$ is full spark.
Then by Lemma~\ref{lem.git separation}(a), $\tau(Z_0)$ is semi-stable under the action of $\operatorname{SL}(d)$.
Proposition~\ref{prop.semistable invariant} then implies that every point in $\operatorname{cl}((\operatorname{SL}(d)\times\mathbb{C}^\times)\cdot \tau(Z_0))$ is also semi-stable under the action of $\operatorname{SL}(d)$.
Importantly, this includes $\tau(Z_\infty)$, where $Z_\infty:=\lim_{t\to\infty}F(Z_0,t)$.
As the limit point of gradient flow, $Z_\infty$ is a critical point of the frame potential, and by Lemma~\ref{lem.git separation}(b), $Z_\infty$ is a minimizer.
It follows that $Z_\infty$ is a unit norm tight frame.
\smartqed
\end{proof}

It remains to prove Lemma~\ref{lem.git separation}, which will require some technology from geometric invariant theory.
For the moment, let's return to the general setting of a reductive group $G$ acting linearly on a finite-dimensional complex vector space $V$.
To discern whether a point is unstable under the action of $G$, it is convenient to work with small subgroups of $G$.
A \textbf{one-parameter subgroup} of $G$ is a homomorphism of algebraic groups $\lambda\colon\mathbb{C}^\times\to G$, meaning there exists a decomposition $V=\bigoplus_{i\in I}V_i$ and integer weights $w\colon I\to\mathbb{Z}$ such that for every $i\in I$, $v\in V_i$, and $t\in\mathbb{C}^\times$, it holds that
\[
\lambda(t)\cdot v
=t^{w(i)}v.
\]
Of course, we may conclude that $v\in V\setminus\{0\}$ is unstable under the action of $G$ if there exists a one-parameter subgroup $\lambda$ of $G$ such that
\[
\lim_{t\to0}\lambda(t)\cdot v=0.
\]
Amazingly, the converse also holds:

\begin{proposition}[Hilbert--Mumford criterion~\cite{Hilbert:93,MumfordFK:94}]
A nonzero vector $v\in V$ is unstable under the action of $G$ if and only if there exists a one-parameter subgroup $\lambda$ of $G$ such that
\[
\lim_{t\to0}\lambda(t)\cdot v=0.
\]
\end{proposition}

\begin{proof}[of Lemma~\ref{lem.git separation}]
(a)
Select any one-parameter subgroup $\lambda\colon\mathbb{C}^\times\to\operatorname{SL}(d)$.
By the determinant, the integer weights $w\colon[d]\to\mathbb{Z}$ of $\lambda$ satisfy $\sum_{i\in[d]}w(i)=0$.
Reindex as necessary so that $w(1)\geq\cdots\geq w(d)$.
If $w(i)=0$ for every $i\in[d]$, then $\lambda(t)$ equals the identity matrix for all $t\in\mathbb{C}^\times$, and so $\lim_{t\to0}\lambda(t)\cdot \tau(Z)=\tau(Z)\neq 0$.
Otherwise, $w(1)>0>w(d)$.
Put $D(t):=\operatorname{diag}\{t^{w(i)}\}_{i\in[d]}$ and take $S\in\operatorname{GL}(d)$ such that $\lambda(t)=SD(t)S^{-1}$.
Given a full spark $Z\in\operatorname{S}(d,n)$, put $\tilde{Z}:=[\tilde{z}_1\cdots\tilde{z}_n]:=S^{-1}Z$.
Then
\begin{align}
S^{-1}\lambda(t)\cdot \tau(Z)
\nonumber
&=D(t)S^{-1}\cdot(z_1\otimes\cdots\otimes z_n)\\
\nonumber
&=D(t)\tilde{z}_1\otimes\cdots\otimes D(t)\tilde{z}_n\\
\nonumber
&=\bigg(\sum_{i_1\in[d]}t^{w(i_1)}(\tilde{z}_1)_{i_1}e_{i_1}\bigg)\otimes\cdots\otimes\bigg(\sum_{i_n\in[d]}t^{w(i_n)}(\tilde{z}_n)_{i_n}e_{i_n}\bigg)\\
\label{eq.find a nonvanishing term}
&=\sum_{i_1,\ldots,i_n\in[d]}\bigg(t^{\sum_{j\in[n]}w(i_j)}\prod_{j\in[n]}(\tilde{z}_j)_{i_j}\bigg)e_{i_1}\otimes\cdots \otimes e_{i_n}.
\end{align}
We will use the fact that $Z$ (and thus $\tilde{Z}$) is full spark to identify a component of \eqref{eq.find a nonvanishing term} that does not vanish as $t\to0$.
By the Hilbert--Mumford criterion, this suffices to prove the result.

Let $k$ denote the number of entries in the $d$th row of $\tilde{Z}$ that equal zero, let $j_1,\ldots,j_k\in[n]$ denote the corresponding column indices, and index the remaining indices as $\{j_{k+1},\ldots,j_n\}:=[n]\setminus\{j_1,\ldots,j_k\}$.
(In the degenerate case where $k=0$, we may take $j_\ell=\ell$ for each $\ell\in[n]$.)
Writing $\tilde{Z}=[\tilde{z}_1\cdots\tilde{z}_n]$, then full spark implies
\[
0
\neq\det([\tilde{z}_{j_1}\cdots\tilde{z}_{j_d}])
=\sum_{\sigma\in S_d}\operatorname{sign}(\sigma)\prod_{t\in[d]}(\tilde{z}_{j_t})_{\sigma(t)}.
\]
It follows that there exists a permutation $\sigma\in S_d$ such that 
\[
\prod_{t\in[d]}(\tilde{z}_{j_t})_{\sigma(t)}
\neq0.
\]
In addition, $(\tilde{z}_{j_t})_d\neq0$ for every $t>k$ by construction, and full spark implies $k<d$.
As such, we consider the component of \eqref{eq.find a nonvanishing term} with indices
\[
i_j
:=\left\{\begin{array}{cl}
\sigma(t)&\text{if $j=j_t$, $t\in[d]$}\\
d&\text{otherwise}
\end{array}\right.
\]
since
\[
\prod_{j\in[n]}(\tilde{z}_j)_{i_j}
=\bigg(\prod_{t\in[d]}(\tilde{z}_{j_t})_{\sigma(t)}\bigg)\bigg(\prod_{t\in\{d+1,\ldots,n\}}(\tilde{z}_{j_t})_d\bigg)
\neq0
\]
and
\[
\sum_{j\in[n]}w(i_j)
=\sum_{i\in[d]}w(i)+(n-d)w(d)
=(n-d)w(d)
<0,
\]
where the last inequality uses the assumption that $d<n$ and the fact that $w(d)<0$.
It follows that the $(i_1,\ldots,i_n)$ component of \eqref{eq.find a nonvanishing term} diverges as $t\to0$.

(b)
Suppose $Z\in\operatorname{S}(d,n)$ is a non-minimizing critical point of the frame potential.
By Corollary~\ref{cor.focp feature}, there exist $d',n'\in\mathbb{N}$ satisfying
\begin{equation}
\label{eq.redundancy inequality}
\frac{n'}{d'}
>\frac{n}{d}
\end{equation}
such that $n'$ of the columns of $Z$ form a tight frame for a $d'$-dimensional subspace of $\mathbb{C}^d$ that is orthogonal to the other $n-n'$ columns.
Explicitly, there is a unitary matrix $U\in\mathbb{C}^{d\times d}$ and a permutation matrix $P\in\mathbb{R}^{n\times n}$ such that
\[
UZP^*
=\left[\begin{array}{cc}
X&0\\
0&Y
\end{array}\right],
\]
where $X\in\mathbb{C}^{d'\times n'}$ satisfies $XX^*=\frac{n'}{d'}I_{d'}$.
Put
\[
D(t):=\left[\begin{array}{cc}
t^{d-d'}I_{d'}&0\\
0&t^{-d'}I_{d-d'}
\end{array}\right],
\]
and take the one-parameter subgroup $\lambda\colon\mathbb{C}^\times\to\operatorname{SL}(d)$ defined by $\lambda(t):=U^{-1}D(t)U$.
Denoting $Z=[z_1\cdots z_n]$ and $UZ=[\tilde{z}_1\cdots\tilde{z}_n]$, then
\begin{align*}
U\lambda(t)\cdot \tau(Z)
&=D(t)U\cdot (z_1\otimes\cdots\otimes z_n)\\
&=D(t)\tilde{z}_1\otimes\cdots\otimes D(t)\tilde{z}_n
=(t^{d-d'})^{n'}(t^{-d'})^{n-n'}\tilde{z}_1\otimes\cdots\otimes \tilde{z}_n.
\end{align*}
The exponent on $t$ simplifies to $n'd-d'n$, which is positive by \eqref{eq.redundancy inequality}.
This implies that $\lim_{t\to0}\lambda(t)\cdot\tau(Z)=0$, as claimed.
\smartqed
\end{proof}

To conclude, we note that in Section~16.3 of the monograph~\cite{Kirwan:online}, Kirwan reports that $\tau(Z)$ is semi-stable under the action of $\operatorname{SL}(d)$ precisely when
\[
\frac{|\{j:z_j\in S\}|}{\operatorname{dim}(S)}
\leq\frac{n}{d}
\]
for every nontrivial subspace $S$ of $\mathbb{C}^d$.
Interestingly, the identical condition appears in Barthe's theorem~\cite{HamiltonM:21} to characterize the real $Z\in\operatorname{S}(d,n)$ for which there exists $A\in\operatorname{GL}(d)$ such that normalizing the columns of $AZ$ gives a unit norm tight frame.
While the original proof of Barthe's theorem seems difficult to parse~\cite{Barthe:98}, this coincidence suggests a geometric proof of the complex case using the fact that gradient flow sends semi-stable points to unit norm tight frames.

\begin{acknowledgement}
Some of the ideas in this project were conceived at the 2018 MFO Mini-Workshop on Algebraic, Geometric, and Combinatorial Methods in Frame Theory.
DGM was partially supported by AFOSR FA9550-18-1-0107 and NSF DMS 1829955.
TN was partially supported by NSF DMS 2107808.
SV was partially supported by NSF DMS 2044349, the NSF--Simons Research Collaboration on the Mathematical and Scientific Foundations of Deep Learning (MoDL) (NSF DMS 2031985), and the TRIPODS Institute for the Foundations of Graph and Deep Learning at Johns Hopkins University.
CS was partially supported by NSF DMS 2107700.
\end{acknowledgement}

\end{document}